\documentclass[12pt,oneside,reqno]{amsart}
\allowdisplaybreaks

\usepackage[top=2.5cm, bottom=2.5cm, left=2.5cm, right=2.5cm]{geometry}
\usepackage{multirow,tabularx}
\newcolumntype{Y}{>{\centering\arraybackslash}X}
\renewcommand\arraystretch{2}
\usepackage{pbox}
\usepackage{parskip}
\usepackage{amssymb}
\usepackage{hyperref}
\usepackage{mathtools}

\usepackage{enumerate}

\newtheorem{thm}{Theorem}[section]
\newtheorem{corollary}[thm]{Corollary}
\newtheorem{lemma}[thm]{Lemma}
\newtheorem{proposition}[thm]{Proposition}

\newtheorem{maintheorem}[thm]{Main Theorem}

\newtheorem{remark}[thm]{Remark}

\newtheorem{property}[thm]{Property}

\makeatletter
\newcommand{\BIG}{\bBigg@{2}}
\newcommand{\vast}{\bBigg@{3}}
\newcommand{\Vast}{\bBigg@{5}}
\makeatother

\numberwithin{equation}{section}

\begin{document}


\baselineskip=17pt



\title[$(\sigma,\tau)$-Derivations of Group Rings]{$\boldsymbol{(\sigma,\tau)}$-Derivations of Group Rings}
\author[D.Chaudhuri]{Dishari Chaudhuri}
\address{Dishari Chaudhuri\\Department of Mathematical Sciences\\ Indian Institute of Science Education and Research Mohali\\
Sector-81, Knowledge City, S.A.S. Nagar, Mohali-140306\\ Punjab, India}
\email{dishari@iisermohali.ac.in, dishari@iitg.ac.in}


\begin{abstract} We study $(\sigma,\tau)$-derivations of a group ring $RG$ of a finite group $G$ over an integral domain $R$ with $1$. As an application we extend a well known result on derivation of an integral group ring $\Bbb{Z}G$ to $(\sigma,\tau)$-derivation on it for a finite group $G$ with some conditions on $\sigma$ and $\tau$. In the process of the extension, a generalization of an application of Skolem-Noether Theorem to derivation on a finite dimensional central simple algebra has also been given for the $(\sigma,\tau)$-derivation case.\\



\end{abstract}

%

\maketitle

\section{Introduction}
Let $R$ be a commutative ring with $1$ and $\mathcal{A}$ be an algebra over $R$. A derivation on $\mathcal{A}$ is an $R$-linear map $\gamma: \mathcal{A}\rightarrow\mathcal{A}$ satisfying $\gamma(ab)=\gamma(a)b+a\gamma(b)$ for all $a,b\in\mathcal{A}$. For $x\in\mathcal{A}$, the derivation $\gamma$ such that $\gamma(a)=xa-ax$ for all $a\in\mathcal{A}$ is called an inner derivation of $\mathcal{A}$ coming from $x$. We will denote such a derivation and inner derivation as $\gamma^{usual}$ and $\gamma_x^{inner}$ respectively. Let $\sigma$, $\tau$ be two different algebra endomorphisms on $\mathcal{A}$. A $(\sigma,\tau)$-derivation on $\mathcal{A}$ is an $R$-linear map $\delta$ satisfying $\delta(ab)\;=\; \delta(a)\tau(b)+\sigma(a)\delta(b)$ for $a,b\in \mathcal{A}$. If $x\in\mathcal{A}$, the $(\sigma,\tau)$-derivation $\delta_x:\mathcal{A}\rightarrow\mathcal{A}$ such that $\delta_{x}(a)=x\tau(a)-\sigma(a)x$, is called a $(\sigma,\tau)$-inner derivation of $\mathcal{A}$ coming from $x$. If $\sigma=\tau=id$, then $\delta$ and $\delta_x$ are respectively the usual derivation $\delta^{usual}$ and inner derivation $\delta_x^{inner}$ of $\mathcal{A}$ coming from $x$. These kinds of derivations were mentioned by Jacobson in \cite{J} (Chapter $7.7$). Later on they have been studied extensively in the case of prime and semiprime rings by many authors. A brief history on such cases can be found in \cite{AAH-06}. Generalized Witt algebras have been studied with the help of $(\sigma,\tau)$-derivations in \cite{HLS-06}, which also contains a study of $(\sigma,\tau)$-derivations on commutative algebras and unique factorization domains. We however are interested in the $(\sigma,\tau)$-derivations of group rings. The usual derivation for integral group rings was studied by Spiegel (Theorem $1$, \cite{Sp-94}). He has shown that for a finite group $G$, every derivation of $\mathbb{Z}G$ is inner. We extend that result to $(\sigma,\tau)$-derivations on $\mathbb{Z}G$ over a finite group $G$ with certain conditions on $\sigma$ and $\tau$. The result we have obtained is for a general group ring $RG$ of a finite group $G$ over any integral domain $R$ with $1$. The integral group ring case then is obtained as a corollary of the main result. Let $\mathcal{Z}(\mathcal{A})$ denote the center of the algebra $\mathcal{A}$. Our main result can be stated as follows:

\begin{maintheorem}\label{main}
 Let $G$ be a finite group and $R$ be an integral domain with $1$ with characteristic $p\geq0$ such that $p$ does not divide the order of $G$.
  \begin{enumerate}[$(i)$]
  \item If $R$ is a field and $\sigma$, $\tau$ are algebra endomorphisms of $RG$ such that they fix $\mathcal{Z}(RG)$ elementwise, then every $(\sigma,\tau)$-derivation of $RG$ is $(\sigma,\tau)$-inner.
  \item If $R$ is an integral domain that is not a field and $\sigma,\;\tau$ are $R$-linear extensions of group homomorphisms of $G$ such that they fix $\mathcal{Z}(RG)$ elementwise, then every $(\sigma,\tau)$-derivation of $RG$ is $(\sigma,\tau)$-inner.
  \end{enumerate}
\end{maintheorem}

Our manuscript has been divided into four sections. The second section contains some properties and interesting results on $(\sigma,\tau)$-derivations. We have provided a generalized corollary of Skolem-Noether Theorem for the $(\sigma,\tau)$-derivation case. The third section is devoted to the proof of our main theorem. In the conclusion part we have mentioned how our result can be applied to $\Bbb{Z}G$.




\section{Useful Results}

The following are some interesting properties of the $(\sigma,\tau)$- derivations which follow directly from the definition. The set of all $(\sigma,\tau)$-derivations on $\mathcal{A}$ will be denoted by $\mathfrak{D}_{(\sigma,\tau)}(\mathcal{A})$.

\begin{property}If $\mathcal{A}$ is unital, then for any $(\sigma,\tau)$-derivation $\delta$, $\delta(1)=0$.
\end{property}

\begin{property}$\mathfrak{D}_{(\sigma,\tau)}(\mathcal{A})$ is an $R$-module as $\delta_1+\delta_2,\;r\delta_1\in\mathfrak{D}_{(\sigma,\tau)}(\mathcal{A})$ for $\delta_1,\delta_2\in\mathfrak{D}_{(\sigma,\tau)}(\mathcal{A})$ and $r\in R$.
\end{property}

\begin{property}When $\sigma(x)a=a\sigma(x)$ $(\text{or } \tau(x)a=a\tau(x))$ for all $x,a\in\mathcal{A}$, and in particular when $\mathcal{A}$ is commutative, $\mathfrak{D}_{(\sigma,\tau)}(\mathcal{A})$ carries a natural left (or right) $\mathcal{A}$-module structure by $(a,\delta)\longmapsto a.\delta:\;x\mapsto a\delta(x)$.
\end{property}


\begin{property}\label{inner sum}For $x,y\in\mathcal{A}$, the $(\sigma,\tau)$-inner derivations satisfy: $\delta_{x+y}=\delta_x+\delta_y$.
\end{property}

\begin{property}\label{sigma inner commutator}For $(\sigma,\tau)$-inner derivations $\delta_x, \delta_y$ for some $x,y\in\mathcal{A}$, $\delta_x=\delta_y$ if and only if $(x-y)\tau(a)=\sigma(a)(x-y)$ for all $a\in\mathcal{A}$.
\end{property}

The following lemma will be useful for the proof of the main theorem.

\begin{lemma}\label{induction}
Let $\sigma$ and $\tau$ be algebra homomorphisms on $\mathcal{A}$ that fix $\mathcal{Z}(\mathcal{A})$ elementwise. Then for a $(\sigma,\tau)$- derivation $\delta$ on $\mathcal{A}$, we have $\delta\left(\alpha^n\right)=n\alpha^{n-1}\delta(\alpha)$ for all $\alpha\in\mathcal{Z}(\mathcal{A})$.
\end{lemma}

\begin{proof}
Let $\alpha\in\mathcal{Z}(\mathcal{A})$. For $n=2$, we have $$\delta\left(\alpha^2\right)=\delta(\alpha)\tau(\alpha)+\sigma(\alpha)\delta(\alpha)=\delta(\alpha)\alpha+\alpha\delta(\alpha)=2\alpha\delta(\alpha)$$. Let the result be true for some $n$, that is, $\delta\left(\alpha^n\right)=n\alpha^{n-1}\delta(\alpha).$
Then $$\delta\left(\alpha^{n+1}\right)=\delta\left(\alpha^n\right)\alpha+\alpha^n\delta\left(\alpha\right)=(n+1)\alpha^n\delta(\alpha).$$ Thus the result follows by induction.
\end{proof}

Recall Skolem-Noether Theorem (for example, Theorem $4.3.1$, \cite{H}).

\begin{thm}[Skolem-Noether]\label{skolem-noether} Let $R$ be a simple Artinian ring with center $F$ and let $A,B$ be simple subalgebras of $R$ which contains $F$ and are finite dimensional over it. If $\phi$ is an isomorphism of $A$ onto $B$ leaving $F$ elementwise fixed, then there is an invertible $x\in R$ such that $\phi(a)=xax^{-1}$ for all $a\in A$.
\end{thm}

We apply the above Skolem-Noether theorem to prove the following result.

\begin{proposition}\label{csa-inner}Let $A$ be a finite dimensional central simple algebra with $1$ over a field $F$ (that is, $A$ is a simple algebra finite dimensional over $F$ such that $\mathcal{Z}(A)=F$). Let $\sigma$ and $\tau$ be non-zero $F$-algebra endomorphisms of $A$. Then there exist units $v_1,v_2$ in $A$ such that any $F$-linear $(\sigma,\tau)$-derivation $\delta$ of $A$ is equal to $v_1\delta_u^{inner}v_2$ for some $u\in A$.
\end{proposition}

\begin{proof}Let $\delta$ be an $F$-linear $(\sigma,\tau)$-derivation of $A$. As $\sigma,\tau$ are non-zero endomorphisms on $A$, by Schur's Lemma, they are actually isomorphisms on $A$. Let $A_2$ be the ring of $2\times2$ matrices over $A$. Then $A_2$ is simple and finite dimensional over $F$ with dimension $4\times[A:F]$. Let
$$B=\left\{
\arraycolsep=2.5pt\def\arraystretch{1}\left(\begin{array}{cc}
                                                                  \sigma(a) & \delta(a) \\
                                                                          0 & \tau(a)
                                                                \end{array}\right)
\middle|a\in A\right\}\qquad\text{and}\qquad C=\left\{\arraycolsep=2.5pt\def\arraystretch{1}\left(\begin{array}{cc}
                                                                  a & 0\\
                                                                          0 & a
                                                                \end{array}\right)
\middle|a\in A\right\}.$$

Then $B$ and $C$ are simple subalgebras of $A_2$. Define a mapping $\Psi:C\rightarrow B$ such that $$\Psi\arraycolsep=2.5pt\def\arraystretch{1}\left(\begin{array}{cc}
                                                                  a & 0\\
                                                                          0 & a
                                                                \end{array}\right)=\arraycolsep=2.5pt\def\arraystretch{1}\left(\begin{array}{cc}
                                                                  \sigma(a) & \delta(a) \\
                                                                          0 & \tau(a)
                                                                \end{array}\right).$$

As $\sigma,\tau$ are $F$-algebra homomorphisms and $\delta$ is additive and $F$-linear, we get $\Psi$ is additive and $F$-linear as well. Also since $\delta(ab)=\delta(a)\tau(b)+\sigma(a)\delta(b)$, we have
\begin{eqnarray*}\Psi\arraycolsep=2.5pt\def\arraystretch{1}\left(\begin{array}{cc}
                                                                  a & 0\\
                                                                          0 & a
                                                                \end{array}\right)\Psi\arraycolsep=2.5pt\def\arraystretch{1}\left(\begin{array}{cc}
                                                                  b & 0\\
                                                                          0 & b
                                                                \end{array}\right)&=&\arraycolsep=2.5pt\def\arraystretch{1}\left(\begin{array}{cc}
                                                                  \sigma(a) & \delta(a) \\
                                                                          0 & \tau(a)
                                                                \end{array}\right)\arraycolsep=2.5pt\def\arraystretch{1}\left(\begin{array}{cc}
                                                                  \sigma(b) & \delta(b) \\
                                                                          0 & \tau(b)
                                                                \end{array}\right)\\
                                                                &=&\arraycolsep=2.5pt\def\arraystretch{1}\left(\begin{array}{cc}
                                                                  \sigma(ab) & \delta(ab) \\
                                                                          0 & \tau(ab)
                                                                \end{array}\right)
                                                                =\Psi\left(\arraycolsep=2.5pt\def\arraystretch{1}\left(\begin{array}{cc}
                                                                  a & 0\\
                                                                          0 & a
                                                                \end{array}\right)\arraycolsep=2.5pt\def\arraystretch{1}\left(\begin{array}{cc}
                                                                  b & 0\\
                                                                          0 & b
                                                                \end{array}\right)\right).\end{eqnarray*}
Thus, $\Psi$ is multiplicative as well. Now, since $\sigma,\;\tau$ are $F$-algebra endomorphisms and so $\sigma(\alpha)=\alpha\sigma(1)=\alpha\tau(1)=\tau(\alpha)=\alpha$ and $\delta(\alpha)=\alpha\delta(1)=0$ for $\alpha\in F$, we get $\Psi$ fixes $F$ elementwise. Also, $\Psi$ is one-one and onto as $\sigma$ and $\tau$ are bijections. Thus $\Psi$ is an isomorphism of $C$ onto $B$ leaving $F$ elementwise fixed. Also, $C$ is isomorphic to $A$, which is simple. So we have $A_2$ is a simple Artinian ring (as it is finite dimensional over its center) with $\mathcal{Z}(A_2)=F$ and $A(\cong C)$, $B$ are simple subalgebras of $A_2$ that are finite dimensional over $F$. Also $\Psi$ is an isomorphism of $A(\cong C)$ onto $B$ leaving $F$ elementwise fixed. Therefore by Theorem \ref{skolem-noether}, there exists an invertible matrix $\left(\begin{smallmatrix}
                                                                                          x & y \\
                                                                                          z & w \\
                                                                                        \end{smallmatrix}\right)\in A_2$ such that
$$
\Psi\arraycolsep=2.5pt\def\arraystretch{1}\left(\begin{array}{cc}
                                                                  a & 0\\
                                                                          0 & a
                                                                \end{array}\right)=\arraycolsep=2.5pt\def\arraystretch{1}\left(\begin{array}{cc}
                                                                  x & y\\
                                                                          z & w
                                                                \end{array}\right)\arraycolsep=2.5pt\def\arraystretch{1}\left(\begin{array}{cc}
                                                                  a & 0\\
                                                                          0 & a
                                                                \end{array}\right){\arraycolsep=2.5pt\def\arraystretch{1}\left(\begin{array}{cc}
                                                                  x & y\\
                                                                          z & w
                                                                \end{array}\right)}^{-1}.
$$ That is,
$$\arraycolsep=2.5pt\def\arraystretch{1}\left(\begin{array}{cc}
                                                                  \sigma(a) & \delta(a) \\
                                                                          0 & \tau(a)
                                                                \end{array}\right)\arraycolsep=2.5pt\def\arraystretch{1}\left(\begin{array}{cc}
                                                                  x & y\\
                                                                          z & w
                                                                \end{array}\right)=\arraycolsep=2.5pt\def\arraystretch{1}\left(\begin{array}{cc}
                                                                  x & y\\
                                                                          z & w
                                                                \end{array}\right)\arraycolsep=2.5pt\def\arraystretch{1}\left(\begin{array}{cc}
                                                                  a & 0\\
                                                                          0 & a
                                                                \end{array}\right).$$
Solving, we get the following set of equations:
\begin{eqnarray}
\label{eqn1}\sigma(a)x+\delta(a)z&=&xa\\
\sigma(a)y+\delta(a)w&=&ya\\
\label{eqn3}\tau(a)z&=&za\\
\tau(a)w&=&wa
\end{eqnarray}
for all $a\in A$. Note that the above equations are symmetric in $w$ and $z$. Since we cannot have $z=w=0$ as then $\left(\begin{smallmatrix}
                                                                                          x & y \\
                                                                                          z & w \\
                                                                                        \end{smallmatrix}\right)$ will not be invertible, we assume one of them is non-zero, say, $z$. As $\sigma,\tau$ are automorphisms on the central simple algebra $A$, by Skolem-Noether Theorem \ref{skolem-noether}, there exist units $s,t$ in $A$ such that $\sigma(a)=s^{-1}as$ and $\tau(a)=t^{-1}at$ for all $a\in A$. From equation \ref{eqn3} we get $t^{-1}atz=za$, that is, $atz=tza$ for all $a\in A$. This implies $tz\in\mathcal{Z}(A)=F$. Let us write $z=t^{-1}\alpha$ for some $\alpha\in F$. Hence, $z^{-1}$ exists. Then from the equation \ref{eqn1}  we get for all $a\in A$,
\begin{eqnarray}
\delta(a)&=&xaz^{-1}-\sigma(a)xz^{-1} \label{eqn delta}\\
\nonumber &=&xa\alpha^{-1}t-s^{-1}asx\alpha^{-1}t\\
\nonumber &=&\alpha^{-1}\left(xat-s^{-1}asxt\right)\\
\nonumber &=&\alpha^{-1}s^{-1}(sxa-asx)t\\
\nonumber &=&v_1(ua-au)v_2
\end{eqnarray} where $v_1=\alpha^{-1}s^{-1},\;v_2=t$ and $u=sx$. So there exist units $v_1,v_2$ in $A$ such that $\delta(a)=v_1\delta_u^{inner}(a)v_2$ for all $a\in A$, where $\delta^{inner}_u$ is the inner derivation of $A$ coming from $u$. Hence proved.
\end{proof}

Another version of the above result is as follows:

\begin{corollary}\label{csa-sigma-inner}Let $A$ be a finite dimensional central simple algebra with $1$ over a field $F$. Let $\sigma$ and $\tau$ be non-zero $F$-algebra endomorphisms of $A$. Then any $F$-linear $(\sigma,\tau)$-derivation of $A$ is a $(\sigma,\tau)$-inner derivation of $A$.
\end{corollary}

\begin{proof}
Following the same steps as Proposition \ref{csa-inner} and taking $\tau(a)=t^{-1}at$ for some invertible $t$, we get finally as in equation \ref{eqn delta},  $\delta(a)=xaz^{-1}-\sigma(a)xz^{-1}$ for all $a\in A$ and where $z=t^{-1}\alpha$, $\alpha\in\mathcal{Z}(A)=F$. We can write for all $a\in A$:

\begin{eqnarray*}
\delta(a)&=&xtt^{-1}att^{-1}z^{-1}-\sigma(a)xz^{-1}\\
&=&xt\tau(a)t^{-1}z^{-1}-\sigma(a)xz^{-1}\\
&=&xt\tau(a)\alpha^{-1}zz^{-1}-\sigma(a)xz^{-1}\\
&=&xz^{-1}\tau(a)-\sigma(a)xz^{-1}\\
&=&u\tau(a)-\sigma(a)u\\
&=&\delta_u(a)
\end{eqnarray*} where $u=xz^{-1}$. Thus a $(\sigma,\tau)$-derivation $\delta$ of $A$ is a $(\sigma,\tau)$-inner derivation of $A$. Hence proved.
\end{proof}


Let us prove some results regarding $(\sigma,\tau)$- derivations on group rings.

\begin{lemma}\label{extension}Let $R$ be a commutative ring with $1$. For algebra homomorphisms $\sigma,\tau: RG\rightarrow RG$, if a map $\delta:G\rightarrow RG$ is such that $\delta(gh)=\delta(g)\tau(h)+\sigma(g)\delta(h)$ for all $g,h\in G$, then $\delta$ extends linearly to a $(\sigma,\tau)$-derivation of $RG$.
\end{lemma}

\begin{proof} Define $\Delta:RG\rightarrow RG$ such that $$\Delta\left(\sum_{g\in G}r_gg\right)=\sum_{g\in G}r_g\delta(g),$$ where $r_g\in R$. Then $\Delta$ is $R$-linear. For $\sum\limits_{g\in G}r_gg,\sum\limits_{h\in G}s_hh\in RG$,
\begin{eqnarray*}
\Delta\left(\sum_{g\in G}r_gg\sum_{h\in G}s_hh\right)&=&\Delta\left(\sum_{g,h\in G}r_gs_hgh\right)\\
&=&\sum_{g,h\in G}r_gs_h\delta(gh)=\sum_{g,h\in G}r_gs_h\big(\delta(g)\tau(h)+\sigma(g)\delta(h)\big)\\
&=&\sum_{g\in G}r_g\delta(g)\sum_{h\in G}s_h\tau(h)+\sum_{g\in G}r_g\sigma(g)\sum_{h\in G}s_h\delta(h)\\
&=&\Delta\left(\sum_{g\in G}r_gg\right)\tau\left(\sum_{h\in G}s_hh\right)+\sigma\left(\sum_{g\in G}r_gg\right)\Delta\left(\sum_{h\in G}s_hh\right).
\end{eqnarray*}
Hence, $\Delta$ is a $(\sigma,\tau)$-derivation of $RG$.
\end{proof}

The following remark is an analogue of the necessary and sufficient conditions for an element being central in a group ring for the $(\sigma,\tau)$-case.

\begin{remark}\label{sigma-commutator}Let $R$ be a commutative ring with $1$ and $\sum\limits_{g\in G}a_gg\in RG$, where $a_g\in R$. Let the algebra homomorphisms $\sigma,\tau: RG\rightarrow RG$ be such that they are $R$-linear extensions of group homomorphisms of $G$ and they fix $\mathcal{Z}(RG)$ elementwise. We obtain the necessary and sufficient condition for an element $a=\sum\limits_{g\in G}a_gg\in RG$ to satisfy the condition $a\tau(b)=\sigma(b)a$ for every $b\in RG$. Note that $\sigma(h)^{-1}g\tau(h)$ is a group element for $g,h\in G$. Thus for any $h\in G$, we get:
\begin{alignat*}{2}
&\quad&\left(\sum\limits_{g\in G}a_gg\right)\tau(h)&=\sigma(h)\left(\sum\limits_{g\in G}a_gg\right)\\
&\Longleftrightarrow &\sum\limits_{g\in G}a_g\sigma(h)^{-1}g\tau(h)&=\sum\limits_{g\in G}a_gg\\
&\Longleftrightarrow \quad &a_g&=a_{\sigma(h)^{-1}g\tau(h)}.
\end{alignat*}
\end{remark}


\section{Proof of Theorem \ref{main}}

Finally, we are now in a position to prove our main theorem. Let $RG$ be the group ring of a finite group $G$ over an integral domain $R$ with $1$ of characteristic $p$ such that $p$ does not divide the order of $G$.\\

\subsection{Proof of part $(i)$}
 
Let $R$ be a field and $\sigma$ and $\tau$ be algebra endomorphisms on $RG$ such that they fix the center of $RG$ elementwise. Let $\delta$ be a $(\sigma,\tau)$-derivation of the group ring $RG$. We need to show that $\delta$ is a $(\sigma,\tau)$-inner derivation of $RG$. \\

  As $G$ is finite and the characteristic of $R$ does not divide the order of $G$, we have $RG$ is semisimple. So by Wedderburn Structure Theorem, $RG$ can be written as a sum of simple algebras in the following way:
$$RG\cong RGe_1\oplus RGe_2\oplus\cdots\oplus RGe_k$$ where
\begin{enumerate}[(i)]
  \item $e_1,e_2,\ldots,e_k$ is a unique collection of idempotents in $RG$ such that they form a central primitive decomposition of $1$.
  \item $RGe_i$ is simple and $RGe_i\cong M_{n_i}(D_i)$ where $D_i$ is a finite dimensional division algebra over $D$ for all $1\leq i\leq k$, and
  \item $\mathcal{Z}(RGe_i)\cong\mathcal{Z}(M_{n_i}(D_i))\cong\mathcal{Z}(D_i)=F_i$, say, such that $F_i$ is a finite dimensional field extension over $R$ for all $1\leq i\leq k$.
\end{enumerate}
Now, let $i,j$ be distinct integers with $1\leq i,j\leq k$. As $e_ie_j=0$, we have
\begin{eqnarray*}
0=\delta(e_ie_j)&=&\delta(e_i)\tau(e_j)+\sigma(e_i)\delta(e_j)\\
&=&\delta(e_i)e_j+e_i\delta(e_j)
\end{eqnarray*}
as $\sigma,\tau$ fix all central elements. Multiplying with $e_j$ from the left, we get $\delta(e_i)e_j=0$ as $e_j$ is a central idempotent and $e_je_i=0$. Thus for each $j$ such that $1\leq j\leq k$ and $i\neq j$, we get $\delta(e_i)e_j=0$, that is, $\delta(e_i)$ is orthogonal to $e_j$. Therefore, we must have $\delta(e_i)\in RGe_i$. Now, if $\alpha\in RG$, then as $\tau(e_i)=e_i$, we have $$\delta(\alpha e_i)=\delta(\alpha)e_i+\sigma(\alpha)\delta(e_i)\in RGe_i.$$
 So we get $\delta(RGe_i)\subseteq RGe_i$. Let us denote $\delta,\;\sigma$ and $\tau$ restricted to $KGe_i$ as $\delta_i,\;\sigma_i$ and $\tau_i$ respectively. Then $\delta_i:RGe_i\rightarrow RGe_i$ is a $(\sigma_i,\tau_i)$-derivation on $RGe_i$. \par
 Now, let $\alpha\in F_i=\mathcal{Z}(D_i)$. Since $\alpha$ is algebraic over $R$, $\alpha$ satisfies a unique monic irreducible polynomial of minimal degree in $R[x]$. Let the polynomial be $p(x)=x^n+a_{n-1}x^{n-1}+\cdots+a_0$. Then $p(\alpha)=\alpha^n+a_{n-1}\alpha^{n-1}+\cdots+a_0=0$. Recall that $\delta(1)=0$. Applying $\delta_i$ to this, we get using Lemma \ref{induction}
$$n\alpha^{n-1}\delta_i(\alpha)+a_{n-1}(n-1)\alpha^{n-2}\delta_i(\alpha)+\cdots+a_1\delta_i(\alpha)=0.$$ That is, $\left[D_x(p(x))\right]_{x=\alpha}\delta_i(\alpha)=0$, where $D_x(p(x))$ denotes the derivative of the polynomial $p(x)$. As $p(x)$ is a minimal polynomial that $\alpha$ satisfies, $\left[D_x(p(x))\right]_{x=\alpha}$ can never be zero. Also it belongs to $F_i$ and so is invertible. This gives us $\delta_i(\alpha)=0$.\par

 Thus if $\beta\in RGe_i$ and $\alpha\in\mathcal{Z}(RGe_i)=F_i$, then $$\delta_i(\alpha\beta)=\delta_i(\alpha)\tau_i(\beta)+\sigma_i(\alpha)\delta_i(\beta)=\sigma_i(\alpha)\delta_i(\beta)=\alpha\delta_i(\beta).$$ Thus, $\delta_i$ is a $(\sigma_i,\tau_i)$-derivation of the finite dimensional simple algebra $RGe_i$ and is linear over $\mathcal{Z}(RGe_i)=F_i$. Hence by Corollary \ref{csa-sigma-inner}, $\delta_i$ is a $(\sigma_i,\tau_i)$-inner derivation of $RGe_i$. Hence by the definition of a $(\sigma_i,\tau_i)$-inner derivation, there exists $x_i\in RGe_i$ such that $\delta_i(a)=\delta_{x_i}(a)=x_ia-ax_i$ for all $a\in RGe_i$. That is, $\delta_i=\delta_{x_i}$. For each $i$, $1\leq i\leq k$, there exists some $x_i\in RGe_i$ such that $\delta_i=\delta_{x_i}$. Then $x=x_1+x_2+\cdots+x_k\in RG$. Then by Property \ref{inner sum}, $\delta=\delta_1+\delta_2+\cdots+\delta_k=\delta_{x_1}+\delta_{x_2}+\cdots+\delta_{x_k}=\delta_x$. Thus, $\delta$ is a $(\sigma,\tau)$-inner derivation on $RG$.\\
 
 \subsection{Proof of Part $(ii)$}
 
 Let $R$ be an integral domain with $1$ that is not a field and $\sigma,\;\tau$ are $R$-linear extensions of group homomorphisms of $G$ such that they fix $\mathcal{Z}(RG)$ elementwise. Let $K$ be the field of fractions of $R$, that is, $K$ is the smallest field containing $R$ such that every nonzero element of $R$ is a unit in $K$.\\
 
 Define $\sigma',\tau':KG\rightarrow KG$ such that $$\sigma'\left(\sum_{g\in G}r_gg\right)=\sum_{g\in G}r_g\sigma(g)$$ and $$\tau'\left(\sum_{g\in G}r_gg\right)=\sum_{g\in G}r_g\tau(g)$$ where $r_g\in K$. Then $\sigma',\tau'$ are $K$-algebra endomorphisms of $KG$. It is known that (see, for example, Lemma $3.6.4$, \cite{SS}), the center $\mathcal{Z}(RG)$ of a group ring $RG$ of a group $G$ over a commutative ring $R$ is the $R$-linear span of finite class sums of $G$, that is, $$\mathcal{Z}(RG)=\left\{\sum_{x\in G}a_xC_x:\;C_x=\sum_{y\in G,y\thicksim x}y,\;a_x\in R\right\},$$ where $y\thicksim x$ for $x,y\in G$ means $y$ is conjugate to $x$. Then, as $\sigma,\tau$ fix $\mathcal{Z}(RG)$ elementwise, we have \[\sigma'\left(\sum_{x\in G}a_xC_x\right)=\sum_{x\in G}a_x\sigma(C_x)=\sum_{x\in G}a_xC_x\] and \[\tau'\left(\sum_{x\in G}a_xC_x\right)=\sum_{x\in G}a_x\tau(C_x)=\sum_{x\in G}a_xC_x.\]
So $\sigma'$ and $\tau'$ fix $\mathcal{Z}(KG)$ elementwise too. Now define $\Delta:KG\rightarrow KG$ such that
$$\Delta\left(\sum_{g\in G}r_gg\right)=\sum_{g\in G}r_g\delta(g),$$ where $r_g\in K$. Then by Lemma \ref{extension}, $\Delta$ is a $(\sigma',\tau')$-derivation of $KG$. By part $(i)$ of Theorem \ref{main}, $\Delta$ is $(\sigma',\tau')$-inner derivation of $KG$. Let $x\in KG$ be such that $\Delta=\Delta_x$, that is, $\Delta(a)=x\tau'(a)-\sigma'(a)x$ for all $a\in KG$.

 Now, every finite integral domain is a field and the case has already been proved in part $(i)$. So let $R$ be an infinite integral domain that is not a field.\\

  Let the element $x$ as obtained above be of the form $x=\sum\limits_{g\in G}x_gg\in KG$, where $x_g\in K$. Let $h\in G$. Then $\Delta(h)=1\delta(h)\in RG$. But $\Delta(h)=\Delta_x(h)$. This means $$\Delta(h)=x\tau'(h)-\sigma'(h)x=x1\tau(h)-1\sigma(h)x=\sum\limits_{g\in G}x_gg\tau(h)-\sum\limits_{g\in G}x_g\sigma(h)g\in RG.$$ Recall that $\sigma$ and $\tau$ are $R$-linear extensions of group homomorphisms of $G$. Now, let $g_i\tau(h)=\sigma(h)g_{k_i}$, for some $g_i,g_{k_i}\in G$, say. Then $\sigma(h)^{-1}g_i\tau(h)=g_{k_i}$. Hence in the above equation, we have $\sum\limits_{g\in G}\left(x_g-x_{\sigma(h)^{-1}g\tau(h)}\right)g\tau(h)\in RG$. That is, $x_g-x_{\sigma(h)^{-1}g\tau(h)}\in R$, where $x_g,x_{\sigma(h)^{-1}g\tau(h)}\in K$. Note that $x_{\sigma(h)^{-1}g\tau(h)}$ is the coefficient of the group element $\sigma(h)^{-1}g\tau(h)$.\\

  Now, we want to find an element $u\in KG$ such that $x-u\in RG$ and $\Delta_x=\Delta_{x-u}$. Then we can conclude that $\Delta$ restricted to $RG$ which is $\delta$, is of the form $\delta_{x-u}$ and hence is a $(\sigma,\tau)$-inner derivation of $RG$. \\


 Let us denote the support of $x$ by $Supp(x)=\{g\;|\; x_g\neq0\}.$ Let $F=\{x_g\;|\;g\in Supp(x)\}$, that is $F$ is the collection of all those elements $x_g\in K$ which appear in the expression of $x$.  As $G$ is finite, we have $F$ is also finite. For $g\in Supp(x)$ we define $$F_{g}=\{x_{\sigma(h)^{-1}g\tau(h)}\in F\;|\;x_g-x_{\sigma(h)^{-1}g\tau(h)}\in R,\; h\in G\}.$$ That is, $F_g$ is the collection of all those elements of $F$ that appear as coefficients of $\sigma(h)^{-1}g\tau(h)$ for $h\in G$ and such that the difference of those elements with $x_g$ belongs to $R$. Note that $x_g$ itself belongs to $F_g$ as for $h=1$, $x_g-x_g=0\in R$. For every $g\in Supp(x)$, we can form the corresponding $F_g$ in this manner. Now, if $F_{g_i}\cap F_{g_j}\neq\emptyset$ for some $g_i,g_j\in Supp(x),\;g_i\neq g_j$, it follows that $F_{g_i}=F_{g_j}$. This is because, if $x_{g_0}\in F_{g_i}\cap F_{g_j}$, then $x_{g_0}-x_{g_i}\in R$ and $x_{g_0}-x_{g_j}\in R$ and this implies $x_{g_i}-x_{g_j}\in R$. Let $g_1,g_2,\cdots, g_t\in Supp(x)$ be the representatives of the distinct $F_g$'s. That is,  $F_{g_1},F_{g_2},\cdots,F_{g_t}$ are distinct subsets of $F$ which means $F_{g_i}\cap F_{g_j}=\emptyset$ for every $1\leq i,j\leq t,\;i\neq j.$ Note that $F=\bigcup_{i=1}^tF_{g_i}$. As R is infinite, so is $K$. Thus for each $F_{g_i}$, $1\leq i\leq t$, which is a finite subset of $K$ we can find an element $x'_{g_i}\in K$ such that $x'_{g_i}\not\in F_{g_i}$ and $x'_{g_i}+R=y+R$ in the abelian group $(K/R,+)$ for every $y\in F_{g_i}$. That is, $x'_{g_i}-y\in R$ for every $y\in F_{g_i}$. Fix this $x'_{g_i}$ for the corresponding $F_{g_i}$.\\

   Define the map: $$((\cdot)):F=\bigcup_{i=1}^{t}F_{g_i}\longrightarrow K$$ where
   $((F_{g_i}))=x'_{g_i}$, that is, $((x_g))=x'_{g_i}$ for every $x_g\in F_{g_i}$, where $x'_{g_i}$ is the element we fixed for $F_{g_i}$ such that $x'_{g_i}\not\in F_{g_i}$ and $x_g-x'_{g_i}\in R$ for every $x_g\in F_{g_i}$.

   Hence for each $x_g\in F$ we have found an element $((x_g))\in K$ such that whenever $x_g-x_{\sigma(h)^{-1}g\tau(h)}\in R$ for $h\in G$, we will have $((x_g))=((x_{\sigma(h)^{-1}g\tau(h)}))$.\\

    Now consider the element $u=\sum\limits_{g\in G}\left((x_g\right))g$ in $KG$. Note that by construction $x-u\in RG$ and for every $h\in G$, $\left((x_g\right))=\left((x_{\sigma(h)^{-1}g\tau(h)})\right)$. Then by Remark \ref{sigma-commutator}, for every $h\in G$ we get $\left(\sum\limits_{g\in G}\left((x_g\right))g\right)\tau(h)=\sigma(h)\left(\sum\limits_{g\in G}\left((x_g\right))g\right).$ We can write it as:

 $$
 \left(\sum_{g\in G}x_gg-\left(\sum_{g\in G}x_gg-\sum_{g\in G}\left((x_g\right))g\right)\right)\tau(h)
 =\sigma(h)\left(\sum_{g\in G}x_gg-\left(\sum_{g\in G}x_gg-\sum_{g\in G}\left((x_g\right))g\right)\right).
 $$

 The above equation can be written as $\left(x-(x-u)\right)\tau(a)=\sigma(a)\left(x-(x-u)\right)$ for every $a\in KG$. Hence, by Property \ref{sigma inner commutator}, we get $\Delta_x=\Delta_{x-u}$. Now, $x-u\in RG$. Thus $\Delta,\;\sigma',\;\tau'$ restricted to $RG$ gives $\delta,\;\sigma,\;\tau$ respectively and we get $$\delta(a)=\delta_{x-u}(a)=(x-u)\tau(a)-\sigma(a)(x-u),$$ for all $a\in RG$. Thus, $\delta$ is a $(\sigma,\tau)$-inner derivation of $RG$. Hence Theorem \ref{main} is proved. \\


 \section{Conclusion}

\noindent Part $(ii)$ of our main theorem can be applied to the integral group ring $\Bbb{Z}G$ of a finite group $G$, with the same conditions on $\sigma$ and $\tau$. Thus we get the following result regarding $(\sigma,\tau)$-derivations of integral group rings which extends Theorem $1$ of \cite{Sp-94} to the $(\sigma,\tau)$-case.

\begin{corollary}
Let $G$ be a finite group and $\sigma,\tau$ be $\Bbb{Z}$-linear extensions of group homomorphisms of $G$ such that they fix $\mathcal{Z}(\mathbb{Z}G)$ elementwise. Then every $(\sigma,\tau)$-derivation of $\mathbb{Z}G$ is $(\sigma,\tau)$-inner.
\end{corollary}

\begin{proof}The result follows directly from our main theorem. Independently also one can prove this corollary by following the same steps as in the proof of our main theorem with a slightly different approach towards the end. Once we get the $(\sigma,\tau)$-derivation $\delta$ of $\Bbb{Z}G$ extended to $(\sigma',\tau')$-derivation $\Delta$ of $\Bbb{Q}G$ is a $(\sigma',\tau')$-inner derivation of $\Bbb{Q}G$ and is of the form $\Delta_x(a)=x\tau(a)-\sigma(a)x$ for all $a\in\Bbb{Q}G$, where $x=\sum_{g\in G}x_gg\in\Bbb{Q}G$, we can construct an element $u$ in $\Bbb{Q}G$ such that $\Delta_x=\Delta_{x-u}$ with $x-u\in \Bbb{Z}G$ by considering the `fractional part' $\{x_g\}$ of each $x_g$. By fractional part $\{x_g\}$ of $x_g$ we mean, $0\leq\{x_g\}<1$ and $\{x_g\}+\Bbb{Z}=x_g+\Bbb{Z}$ in the abelian group $(\Bbb{Q}/\Bbb{Z},+)$. For every $h\in G$ we will get $\sum\limits_{g\in G}\left(x_g-x_{\sigma(h)^{-1}g\tau(h)}\right)g\tau(h)\in \Bbb{Z}G$, that is, $x_g-x_{\sigma(h)^{-1}g\tau(h)}\in \Bbb{Z}$ with $x_g,x_{\sigma(h)^{-1}g\tau(h)}\in \Bbb{Q}$. So we must have $\left\{x_g\right\}=\left\{x_{\sigma(h)^{-1}g\tau(h)}\right\}.$ Then $u=\sum_{g\in G}\{x_g\}g$ will serve our purpose. Thus the $(\sigma,\tau)$-derivation $\delta$ of $\Bbb{Z}G$ with the given conditions on $\sigma$ and $\tau$ is a $(\sigma,\tau)$-inner derivation $\delta_{x-u}$ of $\Bbb{Z}G$.
\end{proof}

\noindent\emph{\textbf{Acknowledgements.}} The author is thankful to DAE (Government of India) and National Board for Higher Mathematics for providing fellowship with reference number $2/40(16)/2016/$ R\&D-II/$5766$. The author would also like to thank IISER Mohali for providing good research facilities when this project was carried out.\\

\bibliographystyle{alpha}
\bibliography{OneBibToRuleThemAll}

\end{document}